\documentclass[12pt,a4paper]{article}
\usepackage[latin1]{inputenc}
\usepackage{amsmath, amsfonts,longtable,mathrsfs,mathtools,float,fontenc, pgfplots, epsfig}
\usepackage{multicol}
\usepackage{multirow}
\usepackage{url}
\usepackage{graphicx}
\usepackage[all]{xy}
\usepackage{amsfonts}
\usepackage{fancyhdr}
\usepackage{verbatim}
\usepackage{amsfonts,longtable,mathrsfs}
\usepackage{epsfig}
\usepackage{amsfonts}
\usepackage{color}
\usepackage[numbers,sort&compress]{natbib}

\date{}
\usepackage{amsthm}
\usepackage{amssymb}
\theoremstyle{plain}

\newtheorem{theorem}{Theorem}[section]
\newtheorem{definition}{Definition}[section]

\numberwithin{equation}{section}
\usepackage{fullpage}
\usepackage{color}
\usepackage{fancyhdr}
\usepackage{verbatim}
\usepackage{amsfonts,longtable,mathrsfs}
\usepackage{epsfig}
\usepackage{amsfonts}
\usepackage{color}
\usepackage[numbers,sort&compress]{natbib}

\def\bc{\begin{center}}
\def\ec{\end{center}}
\def\X{{\cal X}} 

\def\({\left(}
\def\){\right  )}
\def\[{\left[}
\def\]{\right]}
\def\bc{\begin{center}}
\def\ec{\end{center}}
\begin{document}
\bc {\bf Symmetry and Dynamical Analysis of a Discrete Time Model: The Higher Order Berverton-Holt Equation
  }\ec
\medskip
\bc
Mensah Folly-Gbetoula \footnote{Author: Mensah.Folly-Gbetoula@wits.ac.za} 
\\ School of Mathematics, University of the Witwatersrand, Wits 2050, Johannesburg, South Africa.\\
\ec
\begin{abstract}
%\vspace{8cm}
\noindent  In this paper, we study the higher-order Beverton-Holt equation.  We derive non trivial symmetries, and thereafter, solutions are obtained. For constant rate and carrying capacity, we study the periodic nature of the solution and analyze the stability of the equilibrium points have been analyzed.% Moreover, we show that the solution can be expressed in terms of generalized Fibonnaci numbers. 
\end{abstract}
\textbf{2000 Mathematics Subject Classification}: 76M60, 39A05, 39A11\\
\textbf{Keywords}: Difference equation, symmetry, reduction, group invariant, periodic solution
\section{Introduction} \setcounter{equation}{0}
Population studies is a scientific study of animal and human populations. More often, this study involves parameters that are highly connected to the age, gender, geographic distribution as well as the evolution of the population understudy. Many models have been developed: logistic model, the Ricker model,  the Beverton-Holt model, etc. The latter was first introduced in the topic of fisheries by Berverton and Holt in the twentieth century and is known to be  \cite{Bev0}
\begin{equation*}%\label{app}
z_{n + 1} = \frac{\mu K z_{n}}{K + (1-\mu)z_{n}},
\end{equation*}
where $K>0, \mu >1$ are respectively,  the carrying capacity and the inherent growth rate.
\par \noindent
 Later, many authors were attracted by  the periodically forced nonautonomous delay higher-order Beverton-Holt model \cite{Bev}
\begin{align}\label{bohner}
z_{n+k}=\frac{\mu _n K_n z_n}{K_n +(\mu_n -1)z_n}, \qquad n\in \mathbb{N}_0,
\end{align}
where $\mu_n> 1$, $K_n>0$ are non-negative $p$-periodic sequences and the initial conditions   $z_i, i=0, \dots, k-1$, are all positive and interesting results were obtained. 
\par \noindent 
In this work, we study the invariance properties of \eqref{bohner} and we construct its solutions via the invariant of their group of transformations. As one would expect, the study of existence of periodic solutions and the stability character of the solutions becomes easier once  the form of the solutions are known. Higher-order difference equations have been studied from different angles by many researchers \cite{El3, FK1, FK, kgat,Ibrahim, RK, IY}.
\subsection{Preliminaries}
For a deeper knowledge of symmetry analysis of recurrence  equations, one can refer to \cite{hydon0} from which most of our notation and theorems are picked. To start, let $\textbf{x}=(x_1, \dots, x_q)$ be some continuous variables in a given differential equation and $$\mathcal{T}:\textbf{x}\rightarrow \bar{\textbf{x}}(\textbf{x})$$ be a (local) point transformations. 
\begin{definition}
	A parameterized set of transformations
	\begin{equation}\label{Gtransfo}
	\mathcal{T}_{\varepsilon}(\textbf{x})\equiv  \hat{{\textbf{x}}}({\textbf{x}}; \varepsilon)
	\end{equation}
	is a one-parameter local Lie group of transformations if the set of conditions below is  satisfied: 
	\begin{enumerate}
		\item  $\mathcal{T}_0$ is the identity map, so that $\hat{\textbf{x}}=\textbf{x}$ when $\varepsilon =0$.
		\item $\mathcal{T}_{\gamma}\mathcal{T}_{\varepsilon}=\mathcal{T}_{\gamma+\varepsilon}$ for every ${\gamma}, \varepsilon$ sufficiently close to $0$.
		\item Every $\hat{x}_{\alpha}$ can be represented as a Taylor series in $\varepsilon$, that is, $$\hat{x}_{\alpha}(\textbf{x}; \varepsilon)=x_{\alpha} + \varepsilon \eta_{\alpha}(\textbf{x})+O(\varepsilon ^2), \qquad \alpha=0, 1, \dots, q.$$
	\end{enumerate}
\end{definition}
\begin{definition}
	The infinitesimal generator of the one-parameter Lie group of point transformations \eqref{Gtransfo} is the operator
	\begin{align}
	X=X(\textbf{x})=\eta(\textbf{x})\cdot \Delta=\sum_{\alpha=1}^{q}\eta_{\alpha}(\textbf{x})\frac{\partial}{\partial x_{\alpha}},
	\end{align}
	and $\Delta$ is the gradient operator.
\end{definition}
\begin{theorem}
	$F(x)$ is invariant under the Lie group of transformations \eqref{Gtransfo} if and only if
	\begin{align}
	XF(x)=0.
	\end{align}
\end{theorem}
\noindent Consider a forward difference equation
\begin{align}\label{general2}
z_{n+k}=\mathcal{A}(n,z_n, z_{n+1}, \dots, z_{n+k-1}), \qquad n\in D
\end{align}
of order $k$ where $D$ is a regular domain. We strive to find a Lie group of point transformations
\begin{align}\label{transff}
\Phi (n,z_n)=(n, z_n +\varepsilon \zeta(n,z_n)).
\end{align}
Observe that $\varepsilon$ is the group parameter and $\zeta$ is the characteristic of the group of transformations. Let
\begin{align}%\label{Ngener}
G=  Q(n,z_{n})\frac{\partial}{ \partial z_n} + Q(n+1,z_{n+1})\frac{\partial}{ \partial z_{n+1}} +\dots + Q(n+k-1,z_{n+k-1})\frac{\partial}{ \partial z_{n+k-1}} 
\end{align}
be the prolonged generator admitted by the group of point transformations \eqref{transff}. Then the invariance condition reads
\begin{align}\label{LSC}
\begin{split}
&\zeta(n+k,z_{n+k})-\frac{\partial \mathcal{A}}{\partial z_{n+k-1}} \zeta(n+k-1,z_{n+k-1})- \dots - \frac{\partial \mathcal{A}}{ \partial z_n}\zeta(n,z_n)=0
\end{split}
\end{align}
subject to (\ref{general2}). Symmetries are powerful tools for reduction of order of differential and difference equations. In this paper, the reduction of order will be achieved using symmetries and the well-known  canonical coordinate \cite{Maeda}.
\begin{align}
S_n=\int \frac{dz_n}{\zeta(n,z_n)}.
\end{align}
\begin{definition}
	The equilibrium point ${z}$ of \eqref{general2} is stable (locally) if 
	\begin{align}
	\forall \epsilon >0, \exists \; \delta >0\; : \; \sum\limits_{i=0}^{k-1}	|z_i-z|<\delta \implies  
	%	\end{align}
	%	then 
	%	\begin{align}
	|z_n-z|<\epsilon  
	\end{align}
\end{definition}
for all  solutions $\{ z_n\}_{n=0}^{\infty}$ of \eqref{general2}. 
\begin{definition}
	The equilibrium point ${z}$ of \eqref{general2} is a global attractor if  $z_n \rightarrow z$, as 	${n \rightarrow \infty } $, 	for any solution  $\{ z_n\}_{n=0}^{\infty}$ of \eqref{general2}.
\end{definition}
\begin{definition}
	The equilibrium point $z$ of \eqref{general2} is globally asymptotically stable if $z$ is locally stable and it is a global attractor of \eqref{general2}.
\end{definition}
\noindent The polynomial 
\begin{align}\label{ceq}
\lambda^{k}-p_{k-1}\lambda^{k-1}-p_{k-2}\lambda^{k-2}-\dots-p_{1}\lambda^1-p_0=0
\end{align}
where
\begin{align}
p_i =\frac{\partial \mathcal{A}}{\partial z_{n+i}}(z, z,\dots, z, z) 
\end{align}
is referred to as the characteristic equation of  \eqref{general2} near $z$.
\begin{theorem}\label{theo2}
	Suppose $\mathcal{A}$ is a smooth function defined on some neighborhood of $z$. Then,
	\begin{itemize}
		\item[(i)] If all the roots, $\lambda_i$,  of \eqref{ceq} are such that $|\lambda_i|<1$, then the equilibrium point $\bar{x}$ is locally asymptotically stable.
		\item[(ii)] If at least one root of \eqref{ceq} has absolute value greater than one, then the equilibrium point $\bar{x}$ of \eqref{ceq} is unstable.
	\end{itemize}
\end{theorem}
\begin{definition}
	The equilibrium point $z$ of \eqref{general2} is called non-hyperbolic if there exists a root of \eqref{ceq} with absolute value equal to one.
\end{definition}
\begin{theorem}\label{theo3}
	Suppose the $p_i$'s are real numbers satisfying  $$|p_0|+ |p_1|+\dots+|p_{k-1}|<1.$$ Then, the roots of \eqref{ceq} lie inside the open unit disk $|\lambda|<1$.
\end{theorem}
\noindent The above definitions and theorems can be found in \cite{hydon0, ladas}.
\section{Main results}
\subsection{Symmetries}
For the sake of aesthetics, we rewrite the Beverton-holt equation as 
\begin{align}
	z_{n+k}=\frac{z_n}{A_n+B_nz_n},
\end{align}
with 
\begin{align}\label{coef}
	A_n=\frac{1}{\mu_n} \qquad \text{and} \qquad B_n=\frac{\mu_n-1}{K_n \mu_n}.
\end{align}
In the Beverton-Holt model, $\mu_n$ is set to be greater than one. This implies that $A_n$ should be less that one. In this paper, we investigate solutions that are mathematically correct, so without loss of generality, we will will assume that $A_n$ is simply a real number. 
Seeking for Lie symmetries, we force the  criterion of invariance \eqref{LSC} on \eqref{bohner}. This yields
\begin{align}\label{a1}
 &\zeta(n+k,\mathcal{A})-
  {\frac{A_n }{{\left(A_n +B_nz_n\right)}^{2}} } \zeta\left({n, u_{n}}\right)=0.
\end{align}
Solving the functional equation above, we obtain (after a set of lengthy computations) the infinitesimals:
\begin{itemize}
	\item[($i$)]
 \begin{align}\label{a2}
\zeta_1(n, z_n)=&\alpha_n +\frac{B_n}{A_n}\alpha_nz_n 
\end{align}
{ where }  $A_n\alpha_{n+k}-\alpha_n =0$, that is to say,
\begin{align}\label{alpha}
	\alpha_{kn+j}=\alpha_j\left( \prod_{k_1=0}^{n-1}\frac{1}{A_{kk_1+j}}\right), \qquad j=0,\dots, k-1;\end{align} 
	\item[($ii$)] 
\begin{align}\zeta_2(n, z_n)=&\beta_n z_n ^2 
\end{align} 
where $\beta_{n+k} -A_n \beta_n =0$, that is to say,  \begin{align}\label{beta}\beta_{kn+j}=\beta_j\left( \prod_{k_1=0}^{n-1}{A_{kk_1+j}}\right), \qquad j=0,\dots, k-1;\end{align}
\item[($iii$)]
 \begin{align}
\zeta_3(n, z_n)= &\lambda_n z_n + \gamma_n z_n ^2,
\end{align}
{ where }  $\lambda_{n+k}-\lambda_n=0$ \text{ and }  $\gamma_{n+k}-A_n \gamma_n+B_n \lambda_n=0$,  that is to say, \begin{align}\label{lambda} \lambda_n=e^{i\frac{2pn\pi}{k}}\end{align}  and  \begin{align}\label{gamma}\gamma_{kn+j}=\gamma_j\left( \prod_{k_1=0}^{n-1}{A_{kk_1+j}}\right)+\sum_{l=0}^{n-1}B_{kn+j}\lambda_{kn+j}\prod_{k_2=l+1}^{n-1}{A_{kk_1+j}}\end{align} with $p=0,\dots, k-1$ and $j=0,\dots, k-1$.
\end{itemize}
Using these infinitesimals, we obtain the following symmetries:
\begin{align}
X_{1}=&\left( \alpha_n +\frac{B_n}{A_n}\alpha_nz_n \right)\frac{\partial}{\partial z_n}\\X_2=&\beta_n z_n ^2\frac{\partial}{\partial z_n}&\\X_3=&\left( \lambda_n z_n + \gamma_n z_n ^2 \right)\frac{\partial}{\partial z_n}&
\end{align}
where $\alpha_n,\; \beta_n,\; \gamma$ and $\lambda_n$ are given in equations \eqref{alpha}, \eqref{beta}, \eqref{lambda} and \eqref{gamma}, respectively.
\subsection{Canonical coordinate closed form solution}
We select the infinitesimal $\zeta_2$ to lower the order  \eqref{bohner}. Thus, the canonical coordinate takes the form
\begin{align}\label{cano}
S_n =&\int\frac{dz_n}{\zeta_2(n, z_n)}\\=&-\frac{1}{\beta_nz_n}
\end{align}
and therefore
\begin{align}
 S_{n+k}\beta_{n+k}-A_n S_{n}\beta_{n} =& -\frac{1}{z_{n+k}}+\frac{A_n}{z_n}\\
 =&-B_n.
\end{align}
Setting $\tilde{S_n}=-\beta_nS_n$, we find that
\begin{align}\label{tilde}
\tilde{S}_{n+k}= A_n \tilde{S}_n+B_n
\end{align}
and  we take notice of
\begin{equation}\label{inv}
\tilde{S}_n = \frac{1}{z_n}.
\end{equation}
By iterating \eqref{tilde}, we get 
\begin{align}\label{a12}
\tilde{S}_{kn+j}\quad=&\tilde{S}_j \left(   \prod _{k_1=0}^{n-1}A_{kk_1+j}\right) +\sum _{l=0}^{n-1} \left(  B_{kl+j}\prod _{k_2=l+1}^{n-1}A_{kk_2+j}\right),
\end{align}
for $j=0,\dots, k-1$. We reverse the order created by the change of variables to derive the solution of \eqref{bohner}. We observe from \eqref{inv} that
\begin{align*}
z_{n} & = \frac{1}{\tilde{S}_n} 
\end{align*}
from which we obtain
\begin{align}\label{sol}
	z_{kn+j} =& \frac{1}{\tilde{S}_{kn+j}} \nonumber\\
	=&\frac{z_j}{  \left( \prod\limits_{k_1=0}^{n-1}A_{kk_1+j}\right) +z_j\sum\limits _{l=0}^{n-1} \left(  B_{kl+j}\prod\limits _{k_2=l+1}^{n-1}A_{kk_2+j}\right)}, \quad j=0, \dots, k-1.
\end{align}
Recall that $A_n=\frac{1}{\mu_n}$ and $B_n=\frac{\mu_n-1}{K_n \mu_n}$. So,  the closed form solution to the Beverton-Bolt \eqref{bohner} is given by
\begin{align}\label{solmuk}
z_{kn+j} =& \frac{z_j}{  \left( \prod\limits_{k_1=0}^{n-1}\dfrac{1}{\mu_{kk_1+j}}\right) +z_j\sum\limits _{l=0}^{n-1} \left(  \dfrac{\mu_{kl+j}-1}{K_{kl+j} \mu_{kl+j}}\prod\limits _{k_2=l+1}^{n-1}\dfrac{1}{\mu_{kk_2+j}}\right)}, \quad j=0, \dots, k-1.
\end{align}
provided the denominators are non-zero.
\subsection{Periodicity in the growth rate and the carrying capacity} 
The goal of the next section is to investigate the form of the solution when the growth rate and the carrying capacity are 1 or $k$- periodic sequences.
\subsubsection{The case when $\mu_n$ and $K_n$ are $k$-periodic}
\noindent In this case, invoking \eqref{coef}, we have that $A_{k+j}= A_{j}$ and $B_{k+j} = B_{j}$ for all $j$. This reduces \eqref{sol} into (after using the formula for a geometric progression)
\begin{align}\label{solmuknot1}
z_{kn+j} =&\frac{z_j}{  {A_{j}}^n +z_jB_j\left( \frac{1-A_j^n}{1-A_j}\right)}\nonumber\\
=& \frac{(1-A_j)z_j}{ (1-A_j) {A_{j}}^n + B_j\left( {1-{A_j}^n}\right)z_j}, \quad j=0, \dots, k-1
\end{align}
when $A_j\neq1$ and $B_j\neq1$.
\subsubsection{The case when $\mu_n$ and $K_n$ are 1-periodic (constant)}
\noindent Here, thanks to \eqref{coef}, we may assume that $(A_{n})= (A, A, \dots)$ and $B_{n} = (B, B, \dots)$. 
 \begin{itemize}
\item[(i)] When $A = 1$, the solution given in equation \eqref{sol} becomes 
\begin{align}\label{solmuk1c}
z_{kn+j} =&\frac{z_j}{ 1+nBz_j}, \quad j=0, \dots, k-1.
\end{align}
%\noindent where the initial conditions must satisfy $z_j\neq -1/Bn$ slfor all $n$.
 \item[(ii)] {When $A \neq 1$}, the formula solution \eqref{solmuk} reduces to
 \begin{align}\label{solmuknot1c}
 z_{kn+j} =&\frac{z_j}{  {A}^n +z_jB\left( \frac{1-A^n}{1-A}\right)}, \quad j=0, \dots, k-1.
 \end{align}
\end{itemize}
\subsection{Periodicity in the solution and analysis of the stability of the equilibrium points}
In this section, we investigate the existence of  periodic solutions and the nature of the equilibrium points of the model under study. We demonstrate that periodic solutions do exist under certain restrictions.
\begin{theorem}\label{1s}
	The solution $z_n$ of
	%{\scriptsize 
	\begin{equation}\label{eq1}
	z_{n+k}=\frac{z_n}{A_n+B_nz_n},
	\end{equation}%}
	where $A_n\neq 1$, is $k$-periodic if and only  the following conditions are met:
	\begin{itemize}
		\item[(i)]  The sequence $A_n$ and $B_n$ periodic with period $k$.
		\item [(ii)] The initial conditions, $x_j, \; j=0,\dots, k-1$, satisfy $z_j=\frac{1-A_j}{B_j}$.
	\end{itemize}
\end{theorem}
\begin{proof}
	Suppose the initial conditions satisfy the conditions $z_j=\frac{1-A_j}{B_j}$. Using the latter in \eqref{solmuknot1}, we have that:
\begin{align}
	z_{kn+j} =&\frac{z_j}{  {A_{j}}^n +\frac{1-A_j}{B_j}B_j\left( \frac{1-A_j^n}{1-A_j}\right)}\nonumber\\
	=& z_j, \quad j=0, \dots, k-1.
\end{align}	
It follows $z_n$ is periodic with period divisible by $k$. With the choice of the restriction on the initial conditions, the period can not be less that $k$. Therefore, $z_n$ is periodic with period $k$.
\end{proof}
%\noindent Figures \ref{theo1a} and \ref{theo1b} illustrates Theorem \ref{1s}.
\begin{figure}[H]
	\begin{minipage}{0.45\textwidth}
		%	\begin{figure}[h!]
			\centering
			\hspace{-1.3cm}	\includegraphics[scale=0.25]{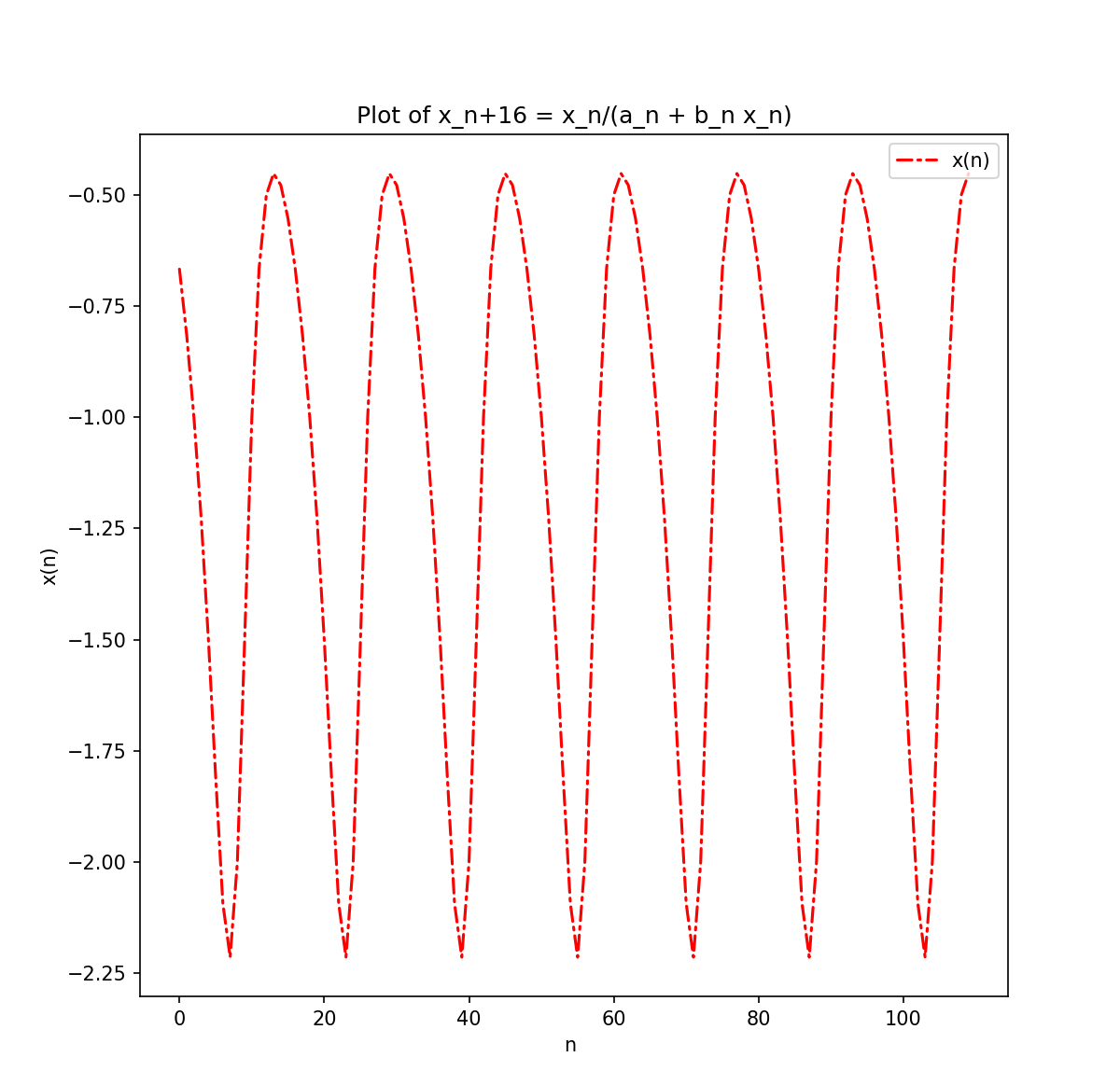}
			\caption{\scriptsize Graph of 
				$z_{n+16} = \dfrac{z_n}{(3+\sin(n\pi/8)+(2+\cos(n\pi/8))z_n}$
				.\label{theo1a}}
			%	\end{figure}
	\end{minipage}\hspace{0.5cm}
	\begin{minipage}{0.45\textwidth}
		\centering
		\includegraphics[scale=0.25]{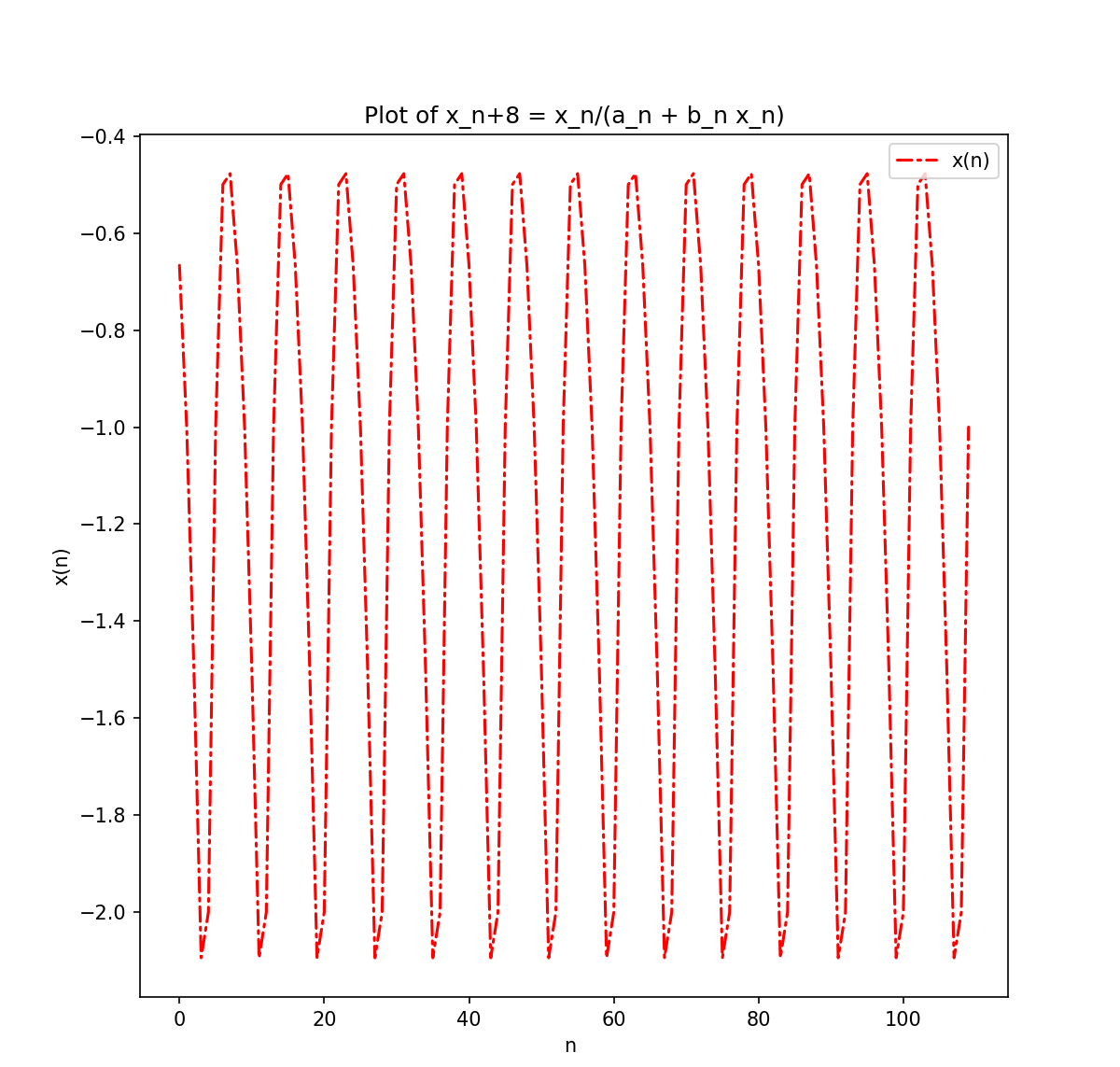}
		\caption{\scriptsize Graph of 
				$z_{n+8} = \dfrac{z_n}{(3+\sin(n\pi/4))+(2+\cos(n\pi/4))z_n}$.\label{theo1b}}
	\end{minipage}
\end{figure} 
\noindent In Figure \ref{theo1a}, we used the initial conditions $$z_j = \frac{-2-\sin(\frac{j\pi}{8})}{2+\cos(\frac{j\pi}{8})},$$ $j=0, \dots, 15$. These initial conditions satisfy the conditions in Theorem \ref{1s}. As predicted, we have $16$-periodic solutions.\par 
\noindent In Figure \ref{theo1b}, we used the initial conditions $$z_j = \frac{-2-\sin(\frac{j\pi}{4})}{2+\cos(\frac{j\pi}{4})},$$ $j=0, \dots, 7$. These initial conditions satisfy the conditions in Theorem \ref{1s}. As predicted, we have $8$-periodic solutions.
\begin{theorem}\label{c1}
	The solution $z_n$ of
	%{\scriptsize 
	\begin{equation}\label{eq1cst}
	z_{n+k}=\frac{z_n}{A+Bz_n},
	\end{equation}%}
	where $A\neq 1$, is $1$-periodic if and only the initial conditions, $x_j, \; j=0,\dots, k-1$, satisfy $z_j=\frac{1-A}{B}$.
\end{theorem}
\begin{proof}
The proof is identical to the proof of Theorem \ref{1s} and is omitted.
\end{proof}
\begin{theorem}\label{2s}
	The solution $z_n$ of
	%{\scriptsize 
	\begin{equation}\label{eq1s}
	z_{n+k}=\frac{z_n}{-1+Bz_n},
	\end{equation}%}
	is $2k$ periodic and contains two cycles of length $k$.
\end{theorem}
\begin{proof}
We recall that the solution of the model for constant coefficients $A_n$ and $B_n$ is given in \eqref{solmuknot1c}. Setting $A=-1$ in \eqref{solmuknot1c}, we get:
\begin{align}
	z_{kn+j} =&\frac{z_j}{  {(-1)}^n +z_jB\left( \frac{1-(-1)^n}{1+1}\right)}, \quad j=0, \dots, k-1.
\end{align}
It follows that
\begin{align}
	z_{2kn+j} =&\frac{z_j}{  {(-1)}^{2n} +z_jB\left( \frac{1-(-1)^{2n}}{1+1}\right)}, \quad j=0, \dots, k-1\\
	=z_j.
\end{align}
In other words,
$$\left\{\dots, z_0, z_1,  \dots,z_{k-1}, \frac{z_0}{-1+z_0B}, \frac{z_1}{-1+z_1B}, \dots, \frac{z_{k-1}}{-1+z_{k-1}B}, z_0, z_1, \dots\right\}$$
are the $2k$ periodic solutions with the two cycles of length $k$.
\end{proof}
\begin{figure}[H]
	\begin{minipage}{0.45\textwidth}
		%	\begin{figure}[h!]
			\centering
			\hspace{-1.3cm}	\includegraphics[scale=0.25]{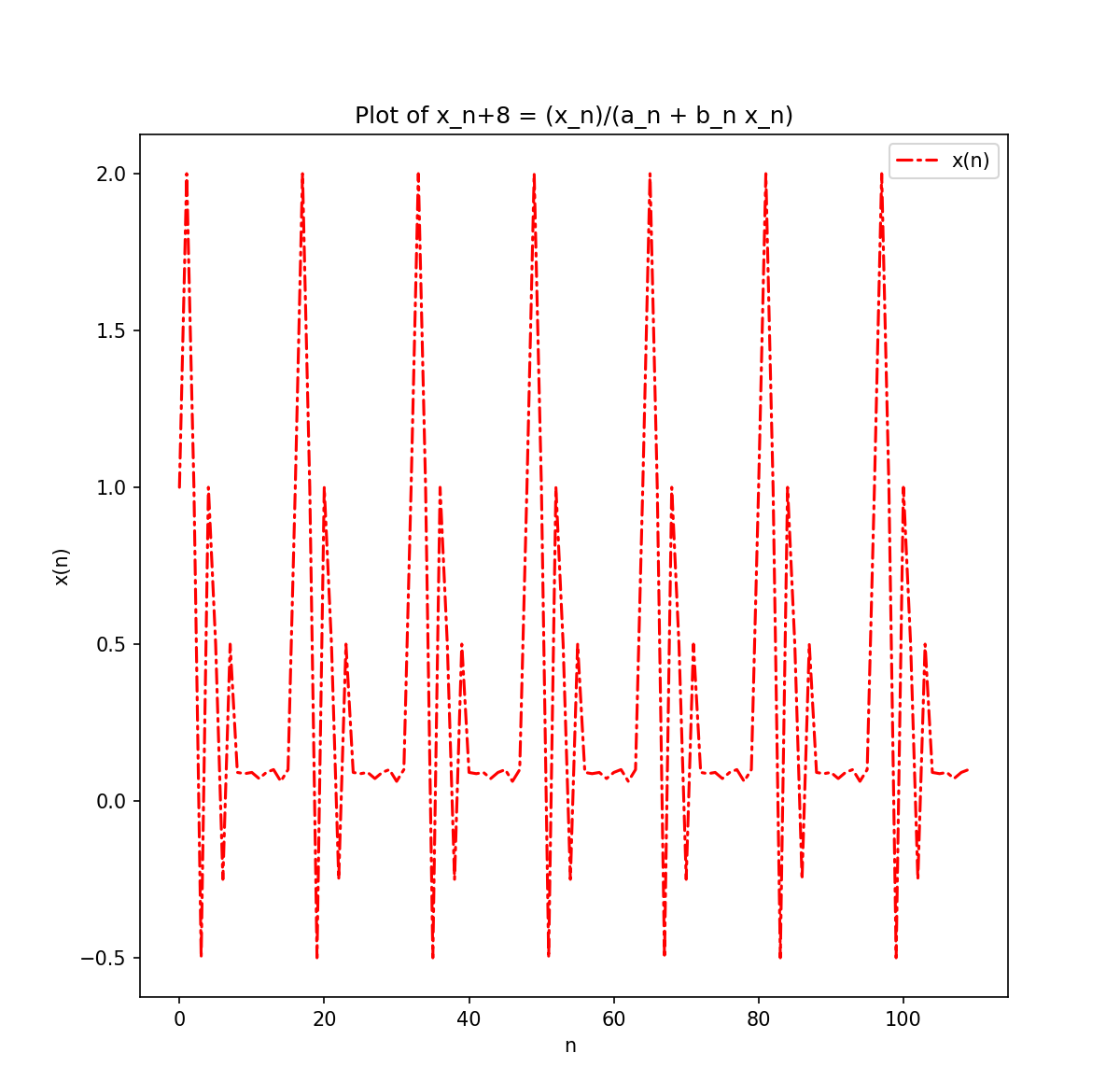}
			\caption{\scriptsize Graph of 
				$z_{n+8} = \dfrac{z_n}{-1+12z_n}$
				.\label{theo3a}}
			%	\end{figure}
	\end{minipage}\hspace{0.5cm}
	\begin{minipage}{0.45\textwidth}
		\centering
		\includegraphics[scale=0.25]{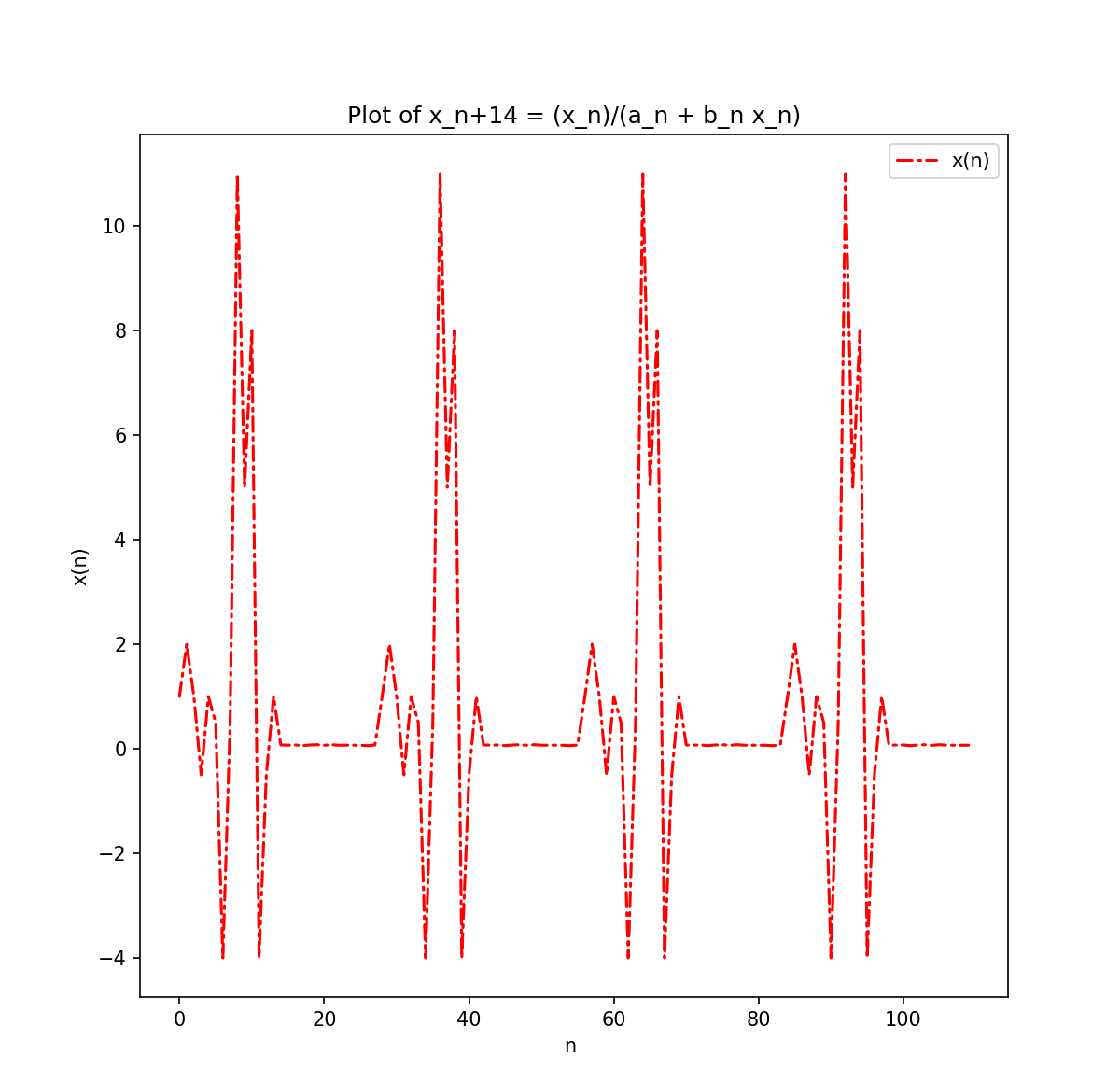}
		\caption{\scriptsize Graph of 
			$z_{n+14} = \dfrac{z_n}{-1+15z_n}$.\label{theo3b}}
	\end{minipage}
\end{figure} 
\noindent In Figure \ref{theo3a}, we used the initial conditions $z_0 = 1; z_1 = 2; z_2 = 1; z_3 = -1/2; z_4 = 1;
z_5 = 1/2; z_6 = -1/4; z_7 = 1/2$.  As predicted, we have $16$-periodic solutions.\par 
\noindent In Figure \ref{theo3b}, we used the initial conditions $z_0 = 1; z_1 = 2; z_2 = 1; z_3 = -1/2; z_4 = 1;
z_5 = 1/2; z_6 = -4; z_7 = 1/2; z_8 = 11; z_9= 5
z_{10} = 8; z_{11} = -4; z_{12} = -1/2; z_{13} = 1$. As predicted, we have $28$-periodic solutions.
%\section{Analysis of the stability of the equilibrium points}
\begin{theorem}
If $A\neq 1$, then the equilibrium point $\bar{z}=0$ of \eqref{bohner} is asymptotically stable when $|A|> 1$ and unstable when $|A|< 1$. Moreover, the non zero equilibrium point $z=(1-A)/B$ of \eqref{bohner} is asymptotically stable when $|A|< 1$ and unstable when $|A|> 1$.   
\end{theorem}
\begin{proof}
	The equilibrium points of \eqref{bohner} are the solutions of the equation  $\bar{z}(A+B\bar{z}-1)=0$. If we let    \begin{align}
		z_{n+k}=	f( z_n)=	\frac{z_n}{A+Bz_n},\end{align}
			 then:
	\begin{itemize}
		\item[-] For $\bar{z}=0$, it is easy to see that $	f_{,z_n}(0)=1/A$ and so, the characteristic equation of \eqref{bohner} around this equilibrium point  is given by $\lambda ^{k}-\frac{1}{A}=0.$ It follows that $|\lambda _i|< 1$ when $|A|> 1$, in order words, $\bar{z}=0$ is locally asymptotically stable. On the other hand, $|\lambda _i |> 1$ when $|A|< 1$ or in order words, $\bar{z}=0$ unstable.
		\item[-] For $\bar{z}=(1-A)/B$,  $f_{,z_n}(0)=A$ and  the characteristic equation of \eqref{bohner} around this equilibrium point is $\lambda ^{k}-A=0$. Consequently, when $|A|< 1$, the roots $\lambda_i$'s of this characteristic equation are such that $\lambda_i|<1$ and therefore the equilibrium is asymptotically stable in this case. Similarly,  when $|A|> 1$, the moduli of the roots are greater than one and therefore the equilibrium is unstable in this case.
	\end{itemize}
\end{proof}
\begin{figure}[H]
	\begin{minipage}{0.45\textwidth}
		%	\begin{figure}[h!]
			\centering
			\hspace{-1.3cm}	\includegraphics[scale=0.25]{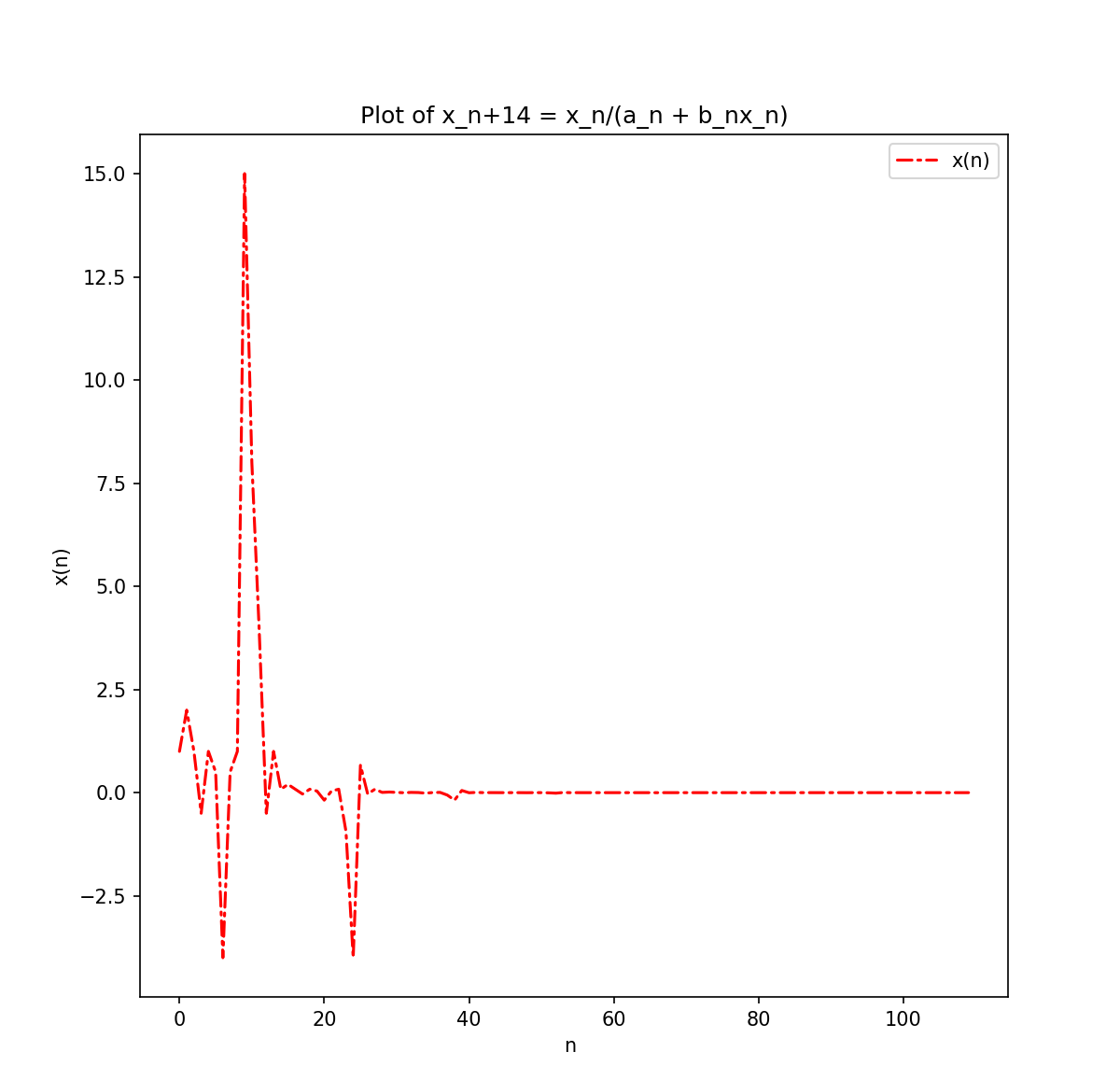}
			\caption{\scriptsize Graph of 
				$z_{n+14} = \dfrac{z_n}{14-2z_n}$
				.\label{theo4a}}
			%	\end{figure}
	\end{minipage}\hspace{0.5cm}
	\begin{minipage}{0.45\textwidth}
		\centering
		\includegraphics[scale=0.25]{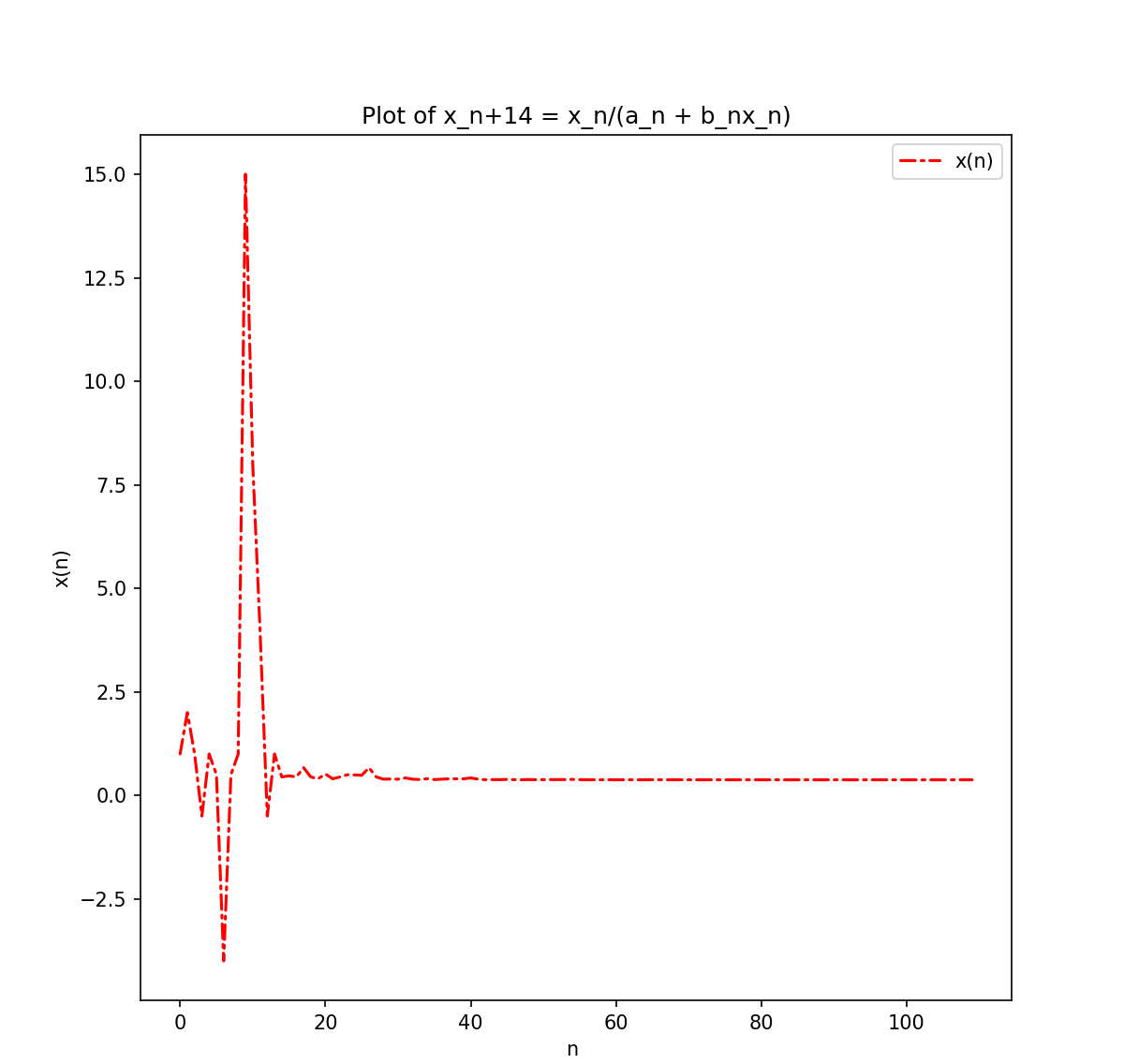}
		\caption{\scriptsize Graph of 
			$z_{n+14} = \dfrac{z_n}{0.25+2z_n}$.\label{theo4b}}
	\end{minipage}
\end{figure} 
\noindent In Figures \ref{theo3a} and \ref{theo3b}, we used the initial conditions $z_0 = 1; z_1 = 2; z_2 = 1; z_3 = -1/2; z_4 = 1;
z_5 = 1/2; z_6 = -4; z_7 = 1/2; z_8=1; z_9=15, z_{10}=8;z_{11}=4; z_{12}=-1/2; z_{13}=1$.
\begin{theorem}
	If $A=1$, then the equilibrium point $\bar{z}=0$ of \eqref{bohner} is non-hyperbolic.   
\end{theorem}
\begin{proof}
	Here, the only equilibrium point of \eqref{bohner}  is $\bar{z}=0$. One can readily check that the characteristic equation of \eqref{bohner} around this equilibrium point  is given by $\lambda ^{k}-1=0.$ There exists a root of $\lambda ^{k}-1=0$  with modulus one. This completes the proof.
	\end{proof}
%%%%%%%%%%%%%%%%%%%%%%%%%%%%%%%%%%%%%%%%5
\section{Conclusion}
\noindent We performed the invariance analysis of the higher-order Beverton-Holt difference equation. Symmetries and the formula solutions are presented. We utilized the canonical coordinate to derive invariants that have been used reduce linearize the equation and eventually obtained the solutions in closed form. We have also presented periodic solutions that satisfy certain ansatz. Lastly, we studied the stability of the equilibrium points of the model.

\end{document}